\documentclass[12pt]{amsart} 
\usepackage{amsmath,amssymb,amsthm} 
\usepackage{color,colordvi,graphicx} 
\newtheorem{theorem}{Theorem} 
\newtheorem{lemma}[theorem]{Lemma} 
 
\newtheorem{proposition}[theorem]{Proposition} 
\newtheorem{algorithm}[theorem]{Algorithm} 
 
\theoremstyle{definition} 
 
\newtheorem{example}[theorem]{Example} 
\newtheorem{remark}[theorem]{Remark} 
 
\newcommand{\CC}{{\mathbb C}}
\newcommand{\NN}{{\mathbb N}}
\newcommand{\QQ}{{\mathbb Q}}
\newcommand{\RR}{{\mathbb R}}
\newcommand{\TT}{{\mathbb T}}
\newcommand{\ZZ}{{\mathbb Z}}
 
\newcommand{\calA}{{\mathcal A}}
\newcommand{\Adot}{{\calA_\bullet}}
\newcommand{\calB}{{\mathcal B}}
\newcommand{\calC}{{\mathcal C}}
\newcommand{\Bdot}{{\calB_\bullet}}

\newcommand{\calV}{{\mathcal V}}

\newcommand{\bzero}{{\boldsymbol 0}}

\newcommand{\vol}{{\rm vol}}
\newcommand{\conv}{{\rm conv}}
\newcommand{\rank}{{\rm rank}}
\newcommand{\Hom}{{\rm Hom}}
\newcommand{\GL}{{\rm GL}}
\newcommand{\MV}{{\rm MV}}

\newcommand{\lhra}{\ensuremath{\lhook\joinrel\relbar\joinrel\relbar\joinrel\rightarrow}}
\newcommand{\lthra}{\ensuremath{\relbar\joinrel\relbar\joinrel\twoheadrightarrow}}

\newcommand{\defcolor}[1]{{\color{blue}#1}} 
\newcommand{\demph}[1]{\defcolor{{\sl #1}}}

\newcommand{\IK}{I}

\newcounter{FNC}[page]
\def\fauxfootnote#1{{\addtocounter{FNC}{2}${\color{magenta}^\fnsymbol{FNC}}$%
     \let\thefootnote\relax\footnotetext{{\color{magenta}$^\fnsymbol{FNC}$#1}}}}
\title{Solving Decomposable Sparse Systems} 
\author[T.~Brysiewicz]{Taylor Brysiewicz} 
\address{T.~Brysiewicz\\ 
         MPI for Mathematics in the Sciences\\
         Leipzig\\ 
         Germany} 
\email{taylor.brysiewicz@mis.mpg.de}
\urladdr{https://sites.google.com/view/taylorbrysiewicz/} 
\author[J.~I.~Rodriguez]{Jose Israel Rodriguez}
\address{J.~I.~Rodriguez\\Department of Mathematics\\
         University of Wisconsin\\
         Madison, WI 53706\\         
         USA}
\email{jRodriguez43@wisc.edu}
\urladdr{http://www.math.wisc.edu/\~{}jose/}
\author[F.~Sottile]{Frank Sottile} 
\address{F.~Sottile\\ 
         Department of Mathematics\\ 
         Texas A\&M University\\ 
         College Station\\ 
         Texas \ 77843\\ 
         USA} 
\email{sottile@math.tamu.edu} 
\urladdr{http://www.math.tamu.edu/\~{}sottile} 
\author[T.~Yahl]{Thomas Yahl} 
\address{T.~Yahl\\ 
         Department of Mathematics\\ 
         Texas A\&M University\\ 
         College Station\\ 
         Texas \ 77843\\ 
         USA} 
\email{thomasjyahl@math.tamu.edu} 
\urladdr{http://www.math.tamu.edu/\~{}thomasjyahl} 
\thanks{Research of Sottile supported by grant 636314 from the Simons Foundation.}
\begin{document} 
 
\begin{abstract}
  Am\'endola et al.\ proposed a method for solving systems of polynomial equations lying in a family which exploits a
  recursive decomposition into smaller systems.
  A family of systems admits such a decomposition if and only if the corresponding Galois group is imprimitive.
  When the Galois group is imprimitive we consider the problem of computing an explicit decomposition.
  A consequence of Esterov's classification of sparse polynomial systems with imprimitive Galois groups is that this
  decomposition is obtained by inspection.
  This leads to a recursive algorithm to compute complex isolated solutions to decomposable sparse systems, which we present and
  give evidence for its efficiency. 
\end{abstract} 
\maketitle 


\section*{Introduction}
The Galois group of a univariate polynomial exposes its internal symmetry and controls its solvability by radicals. 
More generally, families of polynomial systems (and of geometric problems) have Galois groups~\cite{Harris} which expose
their internal symmetry. 
We describe how to solve a polynomial system using numerical homotopy continuation~\cite{Morgan,SW05}
by exploiting the structure of a family to which it belongs.

A family of polynomial systems (geometric problems) is represented as a branched cover of algebraic varieties
$\pi\colon X\to Z$ where $Z$ parameterizes the family and the fiber over $z\in Z$ consists of complex solutions to the
corresponding instance.
Removing the branch locus gives a covering space whose monodromy group is a Galois group~\cite{Harris} of a field
extension. 
Pirola and Schlesinger~\cite{PirolaSchlesinger} observed that the Galois group acts imprimitively if 
and only if after replacing $Z$ by a Zariski open subset $V$, the branched cover factors as a composition
 \begin{equation}\label{Eq:decomposition}
  \pi^{-1}(V)\ \longrightarrow\  Y\ \longrightarrow\  V
 \end{equation}
of nontrivial branched covers, in which case $\pi$ is decomposable.

Am\'endola et al.~\cite{AmendolaRodriguez} explained how to use an explicit decomposition to compute fibers
$\pi^{-1}(z)$ using monodromy~\cite{DHJLLS}.
They showed how several examples in the literature involve a decomposable branched cover.
  In particular, Robert's cognates in kinematics and label swapping in algebraic statistics are used to illustrate the utility of
  decomposability.
  Examples like these span several disciplines and serve as a primary motivation for our study.
For these examples, the variety $Y$
and intermediate maps were determined using invariant theory as there was a finite group acting as automorphisms of
$\pi\colon X\to Z$.
In general, it is  nontrivial to determine a decomposition~\eqref{Eq:decomposition} of a branched cover $\pi\colon X\to Z$
with imprimitive Galois group, especially when the cover admits only the trivial automorphism.

Esterov~\cite{Esterov} determined which systems of sparse polynomials have an imprimitive Galois group.
One goal was to classify those which are solvable by radicals.
He identified two simple structures which imply that the system is decomposable. In these cases, the decomposition is transparent.
He also showed that the Galois group is full symmetric when neither structure occurs.
We use Esterov's classification to give a recursive numerical homotopy continuation algorithm for solving decomposable
sparse systems.

The first such structure is when a polynomial system is composed with a monomial map, such as
$g(x^3)=0$.
To solve this, first solve $g(y)=0$ and then extract third roots of each solution.
The second structure is when the system is triangular, such as
 \[
     f(x,y)\ =\ g(y)\ =\ 0\,.
\]
To solve this, first  solve $g(y)=0$ and then for each solution $y^*$, solve $f(x,y^*)=0$.
The goal of the paper is to recognize and exploit these structures for solving polynomial systems, where
by \demph{solve}, we mean, ``Find all isolated solutions over the complex numbers with nonzero coordinates." 

In general, Esterov's classification leads to a sequence of branched covers, each corresponding to 
a sparse system with symmetric monodromy or to a monomial map.
Our algorithm identifies this structure and uses it to recursively solve a decomposable system.
We give some examples which demonstrate that, despite its overhead, this algorithm is a significant improvement over a
direct use of the polyhedral homotopy~\cite{HuberSturmfels,V99}.

Throughout this paper we assume each  polynomial of a system is prescribed by a finite sum of terms, which consist of a monomial multiplied
by a coefficient.
We use the terminology  \demph{sparse polynomial system} when the monomials of each finite sum are known.
A polynomial system presented as a straight-line program would not be considered sparse, although it could theoretically be translated into one. 
We develop algorithms for solving sparse polynomial systems, which is in comparison to the monodromy methods proposed in
\cite{AmendolaRodriguez} where there is no sparsity requirement. 
Sparsity is important for us because we use the monomial support to identify triangular and lacunary structure.

We say \demph{general sparse polynomial system}, when the coefficients appearing in the sparse system are general.
By the Bernstein-Kushnirenko Theorem~\cite{Bernstein,Kushnirenko}, 
the number of complex isolated solutions to a general sparse system of equations depends only on the convex hulls of the exponent vectors of
the monomials.  
When the system supported on the vertices is decomposable, we propose using it as a start system in a
homotopy to solve the original system.
This is similar in spirit to the B\'ezout or total degree homotopy~\cite{Garcia79}.

In Section~\ref{S:BCGGDP} we present some general background on Galois groups of branched covers and explain the relation
between decompositions of the branched cover and imprimitivity of the Galois group, finishing with a discussion of how to
obtain an explicit decomposition.
We specialize to decomposable sparse systems in Section~\ref{S:DSS}, where we explain Esterov's classification and describe
how to compute the corresponding decompositions.
We present our algorithms for solving sparse decomposable systems in Section~\ref{S:Algorithm}, and give an application to
furnish start systems for homotopies.
Section~\ref{S:computations} gives timings and information on the performance of our algorithm.


\section{Branched covers, Galois groups, and decomposable projections} 
\label{S:BCGGDP}

We sketch some mathematical background, first explaining how Galois groups arise from branched covers and the relationship
between imprimitive Galois groups and decompositions of the branched cover.
We then discuss how to compute a decomposition when the Galois group is imprimitive.

\subsection{Galois groups}

Let $\pi\colon X\to Z$ be a dominant map ($\pi(X)$ is dense in $Z$) of irreducible complex algebraic varieties of the same
dimension.
Such a map is a \demph{branched cover}.
There exists a number $d$ and a
nonempty
Zariski open (in particular, dense, open, and path-connected) subset $U\subset Z$ such that for each $z\in U$, $\pi^{-1}(z)$ consists of $d$
points.
The branched cover is \demph{trivial} when $d=1$.
We define two subgroups of the symmetric group $S_d$ which are well-defined up to conjugacy. 

We may further assume that the map $\pi^{-1}(U)\to U$ is a degree $d$ covering space.
This covering space
has a monodromy group which acts on a fiber $\pi^{-1}(z)$ for $z\in U$ as follows~\cite[Ch.~13]{Munkres}.
Given a loop $\gamma$ in $U$ based at $z$, the lifts of $\gamma$ give $d$ paths in $X$ connecting points of $\pi^{-1}(z)$,
and thus a permutation of $\pi^{-1}(z)$.
The collection of all such monodromy permutations forms the \demph{monodromy group} of $\pi$, which acts transitively
because $\pi^{-1}(U)$ is connected as $X$ is irreducible.

Second, as $\pi\colon X \to Z$ is dominant, the field $\CC(Z)$ of rational functions on $Z$ is a subfield of $\CC(X)$, the field of
  rational functions on $X$.
Since $\pi$ has degree $d$, $\CC(X)$ is a degree $d$ extension of $\CC(Z)$.
If \defcolor{$K$} is the Galois closure of $\CC(X)/\CC(Z)$, then the \demph{Galois group} \defcolor{$G_\pi$} of the
branched cover $\pi\colon X\to Z$ is the Galois group of $K/\CC(Z)$. 
Harris~\cite{Harris} gave a modern proof that the Galois group equals the monodromy group, but this idea goes back
at least to Hermite~\cite{Hermite}.

We recall some terminology concerning permutation groups~\cite{Wielandt}.
Suppose that $G\subset S_d$ is a permutation group acting transitively on the set $\defcolor{[d]}:=\{1,2,\dotsc,d\}$.
A \demph{block} of $G$ is a subset $B\subset [d]$ such that for every $g\in G$, either $gB=B$ or $gB\cap B=\emptyset$.
The subsets $\emptyset$, $[d]$, and every singleton are blocks of every permutation group.
If these trivial blocks are the only blocks, then $G$ is \demph{primitive} and otherwise it is \demph{imprimitive}. 

When $G$ is imprimitive, we have a factorization $d=ab$ with $1<a,b<d$ and there is a bijection
$[a]\times[b]\leftrightarrow[d]$ such that $G$ preserves the projection $[a]\times[b]\to [b]$.
That is, the fibers $\{[a]\times\{i\}\mid i\in [b]\}$ are blocks of $G$,  its action on this set of blocks gives a
homomorphism $G\to S_b$ with transitive image, and the kernel acts transitively on each fiber $[a]\times\{i\}$.
In particular, $G$ is a subgroup of the wreath product
$S_a\wr S_b = (S_a)^b\rtimes S_b$, where $S_b$ acts on $(S_a)^b$ by permuting factors.

We observe a second characterization of imprimitive permutation groups $G$.
Since $G$ acts transitively, if $H\subset G$ is the stabilizer of a point $c\in[d]$, then $H$ has index $d$ in $G$
and we may identify $[d]$ with the cosets $G/H$.
If $B$ is a nontrivial block of $G$ containing $c$, then its stabilizer $L$ is a proper subgroup of $G$ that 
strictly contains $H$.
Furthermore, using the map $G/H \to G/L$, we see that $G$ is imprimitive if and only if the stabilizer of the point
$eH\in G/H$ is not a maximal subgroup.

\subsection{Decomposable branched covers}

A branched cover $\pi\colon X\to Z$ is \demph{decomposable} if there is a nonempty Zariski open subset $V\subset Z$ over which
$\pi$ factors
 \begin{equation}\label{Eq:Ndecomposition}
  \pi^{-1}(V)\ \xrightarrow{\; \varphi\;}\ Y\ \xrightarrow{\; \psi\;}\ V\,,
 \end{equation}
with $\varphi$ and $\psi$ both nontrivial branched covers.
The fibers of $\varphi$ over points of $\psi^{-1}(v)$ are blocks of the action of $G_\pi$ on $\pi^{-1}(v)$, which
implies that $G_\pi$ is imprimitive. 
Pirola and Schlesinger~\cite{PirolaSchlesinger} observed that decomposability of $\pi$ is equivalent to imprimitivity of
$G_\pi$.
We give a proof, as we discuss the problem of computing a decomposition.

\begin{proposition}\label{P:DecomposableIsImprimitive}
 A branched cover is decomposable if and only if its Galois group is~imprimitive.
\end{proposition}  

\begin{proof}
  We need only to prove the reverse direction.
  As above, let $\CC(Z)$, $\CC(X)$, and $K$ be the function fields of $Z$, $X$, and the Galois closure of
  $\CC(X)/\CC(Z)$, respectively, and let $G_\pi$ be the Galois group of $K/\CC(Z)$.
  Let $H$ be the subgroup of $G_\pi$ such that $\CC(X)=K^H$, the fixed field of $H$.
  The set of Galois conjugates of $\CC(X)$ forms the orbit $G_\pi/H$, and the number of conjugates is the degree of
  the branched cover $X\to Z$.

  If $G_\pi$ acts imprimitively, then the stabilizer $L$ of a nontrivial block $B$ containing $\CC(X)$ is a proper subgroup 
  properly containing $H$.
  Thus its fixed field $\defcolor{M}:=K^L$, which is the intersection of the conjugates of $\CC(X)$ in the block $B$, is
  an intermediate field between $\CC(Z)$ and $\CC(X)$.
  For any variety $Y'$ with function field $M$, there will be Zariski open subsets $Y$ of $Y'$ and $V$ of $Z$ such
  that~\eqref{Eq:Ndecomposition} holds.
  Indeed, the inclusions of function fields $\CC(Z)\subset M\subset\CC(X)$ give dominant rational maps $X\dasharrow Y'\dasharrow Z$.
  Replacing the varieties $X$, $Y'$, and $Z$ by Zariski open subsets, we may assume that these are regular maps, hence branched covers.
  Finally, we may replace $Z$ by a nonempty Zariski open subset $V$ contained in the image of $X$ under the composition and let $Y$ be the inverse
  image of $V$ in $Y'$.
\end{proof}
  
  While imprimitivity is equivalent to decomposability, the proof does not address how to compute
  the variety $Y$ of~\eqref{Eq:Ndecomposition}.
  One way is as follows.
  Replace $Z$ and $X$ by affine open subsets, if necessary, and let 
  $y_1,\dotsc,y_m\in\CC[X]$ be regular functions on $X$ that generate $M$ over $\CC(Z)$.
  Let $x_1,\dotsc,x_m$ be indeterminates and let $I\subset\CC(Z)[x_1,\dotsc,x_m]$ be the kernel of the map
  $\CC(Z)[x_1,\dotsc,x_m]\to \CC(X)$ given by $x_i\mapsto y_i$.
  This is the  zero-dimensional ideal of algebraic relations satisfied by $y_1,\dotsc,y_m$.
  Replacing $Z$ by a Zariski open subset of affine space if necessary, we may choose generators $g_1,\dotsc,g_r$ of $I$ that lie in
  $\CC[Z][x_1,\dotsc,x_m]$---their coefficients are regular functions on $Z$.
  There is an open subset $V\subset Z$ such that  the ideal $I$ defines an irreducible variety
  $\defcolor{Y}\subset V\times\CC^m$ whose projection to $V$ is a branched cover and whose function field is $M$.
  Replacing $X$ by $\pi^{-1}(V)$, we obtain the desired decomposition, with the map $\pi^{-1}(V)\to Y$ given by the functions
  $y_1,\dotsc,y_m$. 

  This does not address the practicality of computing $Y$, but it does indicate an approach.
  Given the subgroup $L$ of $G_\pi$ and a set of generators of $\CC[X]$ over $\CC[Z]$, if we apply the Reynolds averaging 
  operator~\cite{DerksenKemper} for $L$ to monomials in the generators, we obtain the desired generators $y_1,\dotsc,y_m$
  of $M$. 
  One problem is that elements of $G_\pi$ may not act on $X$, so their action on elements of $\CC[X]$ may be hard
  to describe.

  There is an exception to this.
  If $L\neq H$ normalizes $H$ in $G$ and $\pi\colon X\to Z$ is a covering space, then $\defcolor{\Gamma}:=L/H$ acts freely
  on $X$, preserving the fibers---it is a group of deck transformations of $X\to Z$~\cite[Ch~13]{Munkres}.
  When $\Gamma$ acts on the original branched cover, $Y=X/\Gamma$ is the desired space, and both $Y$ and the map
  $X\to Y$ may be computed by applying the Reynolds operator for $\Gamma$ to generators of $\CC[X]$.
  The examples given in~\cite[\S~5]{AmendolaRodriguez} are of this form, and the authors use this approach to compute
  the decomposition~\eqref{Eq:Ndecomposition}.

\begin{example}\label{Ex:NoMaximal}
  Not all imprimitive groups have the property that the normalizer $L$ of a point stabilizer $H$ properly contains $H$.
  Consider the wreath product $G:=S_3\wr S_3$, which acts imprimitively on the nine-element set
  $[3]\times[3]$.
  The stabilizer of the point $(3,3)$ is the subgroup $H=((S_3)^2\times S_2)\rtimes S_2$, where $S_2\subset S_3$ is the
  stabilizer of $\{3\}$.
  Then $H$ is its own normalizer in $G$, as $S_2$ is its own normalizer in $S_3$.\hfill$\diamond$
\end{example}

  All imprimitive Galois groups in the Schubert calculus constructed in~\cite[\S~3]{GIVIX} and in~\cite{SWY}
  have the stabilizer $H$ of $\CC(X)$ equal to its normalizer.
  For these, the decomposition of the branched cover follows from a deep structural understanding of
  the corresponding Schubert problem.
  There remain many Schubert problems whose Galois group is expected to be imprimitive, yet we do not 
  know a decomposition~\eqref{Eq:Ndecomposition} of the corresponding branched cover.

  The structure of imprimitivity/decomposability found in~\cite{GIVIX,SWY} was not initially
  apparent, and further study was needed to determine a  decomposition.
  In contrast, a consequence of Esterov's study of Galois groups of sparse polynomial systems is that decomposability is
  transparent and may be deduced by inspection and computing the decomposition~\eqref{Eq:Ndecomposition} is algorithmic.
  This is explained in the following section.

\section{Decomposable Sparse Systems}\label{S:DSS}

We discuss sparse systems of (Laurent) polynomials and interpret them as branched covers.
Then we state the Bernstein-Kushnirenko Theorem for their numbers of complex isolated solutions,
and give the relation between integer linear algebra and maps of algebraic tori.
We then present Esterov's criteria for imprimitivity, and show how these criteria lead to decompositions of
the corresponding branched cover.

\subsection{Sparse Polynomial Systems}

Let $\defcolor{\CC^\times}:=\CC{\smallsetminus}\{0\}$ be the multiplicative group of nonzero complex numbers and
\defcolor{$(\CC^\times)^n$} be the $n$-dimensional complex torus.
For each $\alpha = (\alpha_1,\dotsc,\alpha_n)\in\ZZ^n$, the  \demph{(Laurent) monomial} with exponent $\alpha$,
\[
  \defcolor{x^\alpha}\ :=\ x_1^{\alpha_1} x_2^{\alpha_2}\dotsb x_n^{\alpha_n}\,,
\]
is a character (multiplicative map)  $x^\alpha\colon (\CC^\times)^n\to\CC^\times$.
A finite linear combination 
 \begin{equation}\label{Eq:sparsePolynomial}
   f\ =\ \sum c_\alpha x^\alpha\qquad c_\alpha\in\CC
 \end{equation}
of monomials is a \demph{(Laurent) polynomial}, which is a function $f\colon  (\CC^\times)^n\to\CC$.

The class of sparse polynomial systems pertains to those systems whose monomial structure for each equation is pre-determined. Our polynomial systems naturally occur in a family of sparse polynomial systems determined only by the monomials appearing in each equation of the system.

For a nonempty finite set $\calA\subset\ZZ^n$, the set of polynomials~\eqref{Eq:sparsePolynomial} satisfying
$c_\alpha\neq 0 \Rightarrow \alpha\in\calA$ is the vector space \defcolor{$\CC^\calA$} of polynomials of
\demph{support} $\calA$.
Given a collection $\defcolor{\Adot}:=(\calA_1,\dotsc,\calA_n)$ of nonempty finite subsets of $\ZZ^n$, write
$\defcolor{\CC^\Adot}:=\CC^{\calA_1}\times\dotsb\times\CC^{\calA_n}$ for the vector space of $n$-tuples $F=(f_1,\dotsc,f_n)$ of
polynomials, where $f_i$ has support $\calA_i$, for each $i$.
An element $F\in\CC^\Adot$ is a list of coefficients of these polynomials, which corresponds to a system of polynomial equations
 \[
     f_1(x_1,\dotsc,x_n)\ =\ 
     f_2(x_1,\dotsc,x_n)\ =\ \dotsb\ =\ 
     f_n(x_1,\dotsc,x_n)\ =\ 0\,,
 \]
written $F(x) = 0$. Such a system of polynomial equations is called a  \demph{sparse polynomial system} of \demph{support  $\Adot$}.
Its set of solutions in $(\CC^\times)^n$ is \defcolor{$\calV(F)$}.

Given supports $\Adot=(\calA_1,\dotsc,\calA_n)$, consider the incidence variety
\[
  \defcolor{X_{\Adot}}\ :=\ \left\{(F,x)\in\CC^{\Adot}\times(\CC^\times)^n \mid F(x) = 0  \right\}.
\]
equipped with projections $\defcolor{\pi}\colon X_{\Adot}\to\CC^{\Adot}$ and
$\defcolor{p}\colon X_{\Adot}\to(\CC^\times)^n$.
For any point $x\in(\CC^\times)^n$, the fiber $p^{-1}(x)$ is a vector subspace of $\CC^{\Adot}$ of codimension $n$.
Indeed, for each $i=1,\dotsc,n$, the condition that $f_i(x)=0$ is a linear equation in the coefficients $\CC^{\calA_i}$ of 
$f_i$, and these $n$ linear equations are independent.
Thus $X_{\Adot}$ is irreducible of dimension
\[
  \dim(\CC^\times)^n + \dim\CC^{\Adot} - n\ =\ \dim\CC^{\Adot}\,.
\]

For $F\in\CC^{\Adot}$, the fiber $\pi^{-1}(F)$ is the set $\calV(F)$ of solutions in $(\CC^\times)^n$ to $F(x)=0$.
The image of $X_{\Adot}$ under $\pi$ either lies in a proper subvariety \defcolor{$Z$} of $\CC^{\Adot}$ or it is dense in
$\CC^{\Adot}$. 
In the first case, there is a Zariski open subset $\defcolor{U}:=\CC^{\Adot}\smallsetminus Z$ consisting of polynomial
systems $F(x)=0$ with no solution.
In the second case, $\pi\colon X_{\Adot}\to\CC^{\Adot}$ is a branched cover, so there is a positive integer $d$ and a Zariski
open subset $U\subset\CC^{\Adot}$ consisting of polynomial systems $F(x)=0$ with $d$ isolated solutions.
Both cases are determined by the polyhedral geometry of the supports $\Adot$ through the Bernstein-Kushnirenko Theorem.

For convex sets $K_1,\dotsc,K_n\subset\RR^n$ and nonnegative real numbers, $t_1,\dotsc,t_n\in\RR_{\geq0}$, Minkowski proved
that the volume of the Minkowski sum
\[
  t_1K_1+\dotsb+t_nK_n\ :=\ \{ t_1 x_1+\dotsb + t_nx_x \mid x_i\in K_i\}
\]
is a homogeneous polynomial of degree $n$ in $t_1,\dotsc,t_n$.
Its coefficient of $t_1\dotsb t_n$ is the \defcolor{mixed volume} of $K_1,\dotsc,K_n$.
For $\Adot=(\calA_1,\dotsc,\calA_n)$, let \defcolor{$\MV(\Adot)$} be the mixed volume of
$\conv(\calA_1),\dotsc,\conv(\calA_n)$, where  \defcolor{$\conv(\calA_i)$} is the convex hull of $\calA_i$.  
This is described in more detail in~\cite[Sect.~IV.3]{Ewald}.
We give the Bernstein-Kushnirenko Theorem~\cite{Bernstein,Kushnirenko}.

\begin{proposition}\label{prop:BKK}
  Let $F(x)=0$ be a system of polynomials with support $\Adot$.
  The number of complex isolated solutions in $(\CC^\times)^n$ to $F(x)=0$ is at most $\MV(\Adot)$.
  There is a Zariski open subset $U\subset\CC^{\Adot}$ consisting of systems with exactly $\MV(\Adot)$ solutions.
\end{proposition}

Thus $\pi\colon X_{\Adot}\to\CC^{\Adot}$ is a branched cover if and only if $\MV(\Adot)\neq 0$, which was determined by
Minkowski as follows.
For a nonempty subset $I\subseteq\defcolor{[n]}:=\{1,\dotsc,n\}$, write $\defcolor{\calA_I}:=(\calA_i\mid i\in I)$ and
\defcolor{$\ZZ\calA_I$} be the affine span of the supports in $\calA_I$.
This is the free abelian group generated by differences $\alpha-\beta$ for $\alpha,\beta\in \calA_i$ for some $i\in I$.
Then $\MV(\Adot) = 0$ if and only if there exists a nonempty subset $I\subseteq[n]$ such
that $|I|$ exceeds $\rank(\ZZ\calA_I)$.
In particular, $\MV(\Adot)\neq 0$ implies that $\defcolor{\ZZ\Adot}=\ZZ\calA_{[n]}$ has full rank $n$.

The branched cover $\pi\colon X_{\Adot}\to\CC^{\Adot}$ is nontrivial 
when  $\MV(\Adot)>1$.
Given supports $\Adot$ with  $\MV(\Adot)\neq 0$, let \defcolor{$G_\Adot$} be the Galois group of the branched cover
$X_{\Adot}\to\CC^{\Adot}$.

\subsection{Integer linear algebra and coordinate changes}\label{SS:CoordinateChanges}

As a monomial $x^\beta$ for $\beta\in\ZZ^n$ is invertible on $(\CC^\times)^n$, polynomials $f$ and $x^\beta f$ have the
same sets of zeroes.
If $\calA$ is the support of $f$, then the support of $x^\beta f$ is $\beta+\calA$, the translation of $\calA$ by $\beta$. 
Thus translating the supports of sparse polynomials by integer vectors does not change any assertions about their
zeroes in $(\CC^\times)^n$.
Similarly, $\ZZ\calA=\ZZ(\beta+\calA)$.
Consequently, we will henceforth assume that ${\bf{0}}\in\calA$, for then $\ZZ\calA$ will be the $\ZZ$-linear span
of $\calA$.

We identify the set $\Hom((\CC^\times)^n,\CC^\times)$ of characters on $(\CC^\times)^n$ with the free abelian group
$\ZZ^n$.
A group homomorphism $\Phi\colon(\CC^\times)^m\to(\CC^\times)^k$ is determined by $k$ characters of $(\CC^\times)^m$,
equivalently by a homomorphism (linear map) $\varphi\colon\ZZ^k\to\ZZ^m$ of free abelian groups---$\varphi$ is also the map
pulling a character of $(\CC^\times)^k$ back along $\Phi$.
In particular, an invertible map $\Phi\colon(\CC^\times)^n\to(\CC^\times)^n$ (a monomial change of coordinates) pulls back to an
invertible map $\varphi\colon\ZZ^n\to\ZZ^n$, identifying $\GL(n,\ZZ)$ with the group of possible monomial coordinate changes.
We will write $\Phi=\varphi^*$   and $\varphi=\Phi^*$  for these. 
If $\Phi(x)=(x^{\alpha_1},\dotsc,x^{\alpha_n})$, then the map $\varphi=\Phi^*\colon\ZZ^n\xrightarrow{\,\sim\,}\ZZ^n$ sends the $i$-th standard
basis vector $e_i$ to $\alpha_i$ and is represented by the invertible matrix $A$ whose $i$-th column is $\alpha_i$. When the integer span of
$\alpha_1,\dotsc,\alpha_n$ is $\ZZ^n$, the map $\varphi = \Phi^*$ is invertible. 

Suppose that $f$ is a polynomial on $(\CC^\times)^n$ with support $\calA$.
Given a homomorphism $\Phi\colon(\CC^\times)^m\to(\CC^\times)^n$, the composition $f(\Phi(z))$ for $z\in(\CC^\times)^m$ is
a polynomial with support $\varphi(\calA)$, where the coefficient of $z^\beta$ is the sum of coefficients of $x^\alpha$ for
$\alpha\in \varphi^{-1}(\beta)\cap\calA$.

\subsection{Decompositions of Sparse Polynomial Systems}\label{SS:Decompositions}
We describe two properties that a collection $\Adot$ of supports may have, lacunary and (strictly) triangular, and then
recall Esterov's theorem about the Galois group $G_\Adot$.
We then present explicit decompositions of the projection $\pi\colon X_{\Adot}\to \CC^{\Adot}$ when $\Adot$ is lacunary
and when $\Adot$ is triangular.
These form the basis for our algorithms.

Let $\Adot=(\calA_1,\dotsc,\calA_n)$ be a collection of supports.
Assume that $\MV(\Adot)>1$.
We say that $\Adot$ is \demph{lacunary} if the affine span $\ZZ\Adot\neq\ZZ^n$ (it has rank $n$ as $\MV(\Adot)\neq 0$).
We say that $\Adot$ is \demph{triangular} if there is a nonempty proper subset $\emptyset\neq I\subsetneq[n]$ such that
$\rank(\ZZ\calA_I)=|I|$.
As we explain in Section~\ref{S:SNF}, we may change coordinates and assume that $\ZZ\calA_I\subset\ZZ^{|I|}$ so that
$\MV(\calA_I)$ is defined using $\conv(\calA_i)\subset\RR^{|I|}$ for $i\in I$.
A system $\Adot$ of triangular supports is \demph{strictly triangular} if for some $\emptyset\neq I\subsetneq[n]$ with 
$\rank(\ZZ\calA_I)=|I|$, we have $1<\MV(\calA_I)<\MV(\Adot)$.
It is elementary that if  $\Adot$ is either lacunary or strictly triangular, then the
branched cover $X_\Adot\to\CC^\Adot$ is decomposable and therefore $G_\Adot$ is an imprimitive permutation group.
We show this explicitly in Sections \ref{SSS:Lacunary} and \ref{SSS:triangular}. 
  
\begin{proposition}[Esterov~\cite{Esterov}]
  Let $\Adot$ be a collection of supports with $\MV(\Adot)\neq 0$.
  The Galois group $G_\Adot$ is equal to the symmetric group $S_{\MV(\Adot)}$ if and only if $\Adot$ is neither lacunary
  nor strictly triangular.
\end{proposition}

\subsubsection{Lacunary support}\label{SSS:Lacunary}

Let us begin with an example when $n=2$.
We represent vectors by the columns of a matrix.
Let
\[
  \calA_1\ :=\ \left(\begin{matrix} 0&0&3&6&12\\0&4&3&6&0 \end{matrix}\right)
   \qquad\mbox{and}\qquad
  \calA_2\ :=\ \left(\begin{matrix} 0&3&6&9&9\\0&7&2&1&5 \end{matrix}\right)
\]
be supports in $\ZZ^2$.
Then $\ZZ\Adot$ has index 12 
in $\ZZ^2$ as the map $\varphi(a,b)^T=(3a,4b-a)^T$ is an isomorphism
$\varphi\colon\ZZ^2\xrightarrow{\,\sim\,}\ZZ\Adot$, and $\det(\begin{smallmatrix}3&0\\-1&4\end{smallmatrix})=12$.
If we set $\defcolor{\calB_i}:=\varphi^{-1}(\calA_i)$, then 
\[
  \calB_1\ :=\ \left(\begin{matrix} 0&0&1&2&4\\0&1&1&2&1 \end{matrix}\right)
   \qquad\mbox{and}\qquad
  \calB_2\ :=\ \left(\begin{matrix} 0&1&2&3&3\\0&2&1&1&2 \end{matrix}\right)\ .
\]
We display $\calA_1$, $\calA_2$, $\calB_1$, and $\calB_2$ below.
\[
  \begin{picture}(136,93)(0,-11)
    \put(0,0){\includegraphics{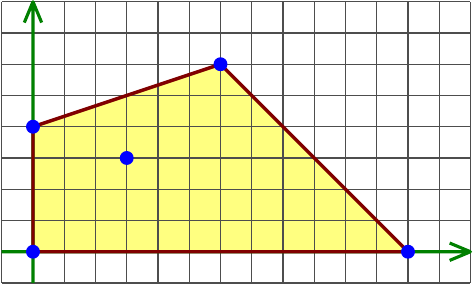}}
    \put(52,-11){$\calA_1$}
  \end{picture}
  \quad
  \begin{picture}(100,93)(0,-11)
    \put(0,0){\includegraphics{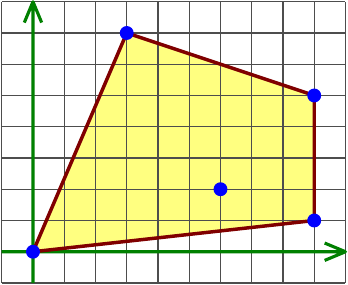}}
    \put(43,-11){$\calA_2$}
  \end{picture}
  \qquad
  \begin{picture}(73,66)(0,-11)
    \put(0,0){\includegraphics{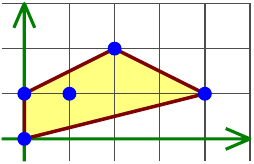}}
    \put(32,-11){$\calB_1$}
  \end{picture}
  \quad
  \begin{picture}(60,66)(0,-11)
    \put(0,0){\includegraphics{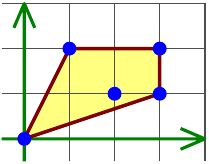}}
    \put(26,-11){$\calB_2$}
  \end{picture}
\]
Then the map $\Phi:=\varphi^*\colon(\CC^\times)^2\twoheadrightarrow(\CC^\times)^2$ is given by
$\Phi(x,y)=(x^3y^{-1},y^4)=(z,w)$.
If 
\begin{align*}
  f_1\ &=\ 1+2y^4+4x^3y^3+8x^6y^6+16x^{12}\\
  f_2\ &=\ 3+5x^3y^7+7x^6y^2+11x^9y+13x^9y^5\ ,
\end{align*}
which is a polynomial system with support $\Adot$, then $f_i=g_i\circ\Phi$, where  
 \begin{align*}
  g_1\ &=\ 1+2w+4zw+8z^2w^2+16z^4w\\
  g_2\ &=\ 3+5zw^2+7z^2w+11z^3w+13z^3w^2\,,
\end{align*}
is a polynomial system with support $\Bdot$. 
Thus the branched cover $X_{\Adot}\to\CC^{\Adot}$ factors  $X_{\Adot}\to X_{\Bdot}\to\CC^{\Bdot}=\CC^{\Adot}$
with the map $X_{\Adot}\to X_{\Bdot}$ induced by $\Phi$.
This implies that $G_{\Adot}\subset (\ZZ/12\ZZ)^{10}\rtimes S_{10}$, as $\ZZ^2/\ZZ\Adot\simeq\ZZ/12\ZZ$, $\Bdot$ is neither
lacunary nor triangular, and $\MV(\Bdot)=10$.

We generalize this example.
Suppose  that $\Adot=(\calA_1,\dotsc,\calA_n)$ is lacunary.
Then $\ZZ\Adot$ has rank $n$ but $\ZZ\Adot\neq\ZZ^n$.
Let $\varphi\colon\ZZ^n\xrightarrow{\,\sim\,}\ZZ\Adot$ be an isomorphism.
Then the corresponding map $\Phi=\varphi^*\colon(\CC^\times)^n\to(\CC^\times)^n$ 
is a surjection with kernel $\Hom(\ZZ^n/\ZZ\Adot,\CC^\times)$.
For each $i=1,\dotsc,n$, set $\defcolor{\calB_i}:=\varphi^{-1}(\calA_i)$.
Then $\Bdot=(\calB_1,\dotsc,\calB_n)$ is a collection of supports with $\ZZ\Bdot=\ZZ^n$.
Since $\varphi$ is a bijection, we identify $\CC^{\calB_i}$ with $\CC^{\calA_i}$ and $\CC^{\Bdot}$ with
$\CC^{\Adot}$.
Given a system $F(x)=0$ with $F\in\CC^{\Adot}$, $\iota(F)(x)=0$ with $\defcolor{\iota(F)}\in\CC^{\Bdot}$ is the corresponding system with support
$\Bdot$. 

\begin{lemma}\label{L:Lacunary}
 Suppose that $\Adot$ is lacunary, $\varphi\colon\ZZ^n\xrightarrow{\,\sim\,}\ZZ\Adot$ is an
 isomorphism with corresponding surjection  $\Phi\colon(\CC^\times)^n\to(\CC^\times)^n$.
 Let $\Bdot:=\varphi^{-1}(\Adot)$ and suppose that $\MV(\Bdot)>1$.
 Then the branched cover $X_\Adot\to\CC^\Adot$ is decomposable and
 $X_\Adot\to X_\Bdot\to\CC^\Adot=\CC^\Bdot$ is a nontrivial decomposition of branched covers induced by the map $\Phi$.
\end{lemma}

\begin{proof}
If $g$ is a polynomial with support $\calB\subset\ZZ^n$, then the composition $g\circ\Phi$ is a polynomial with support
$\varphi(\calB)$, with the coefficient of $x^\beta$ in $g$ equal to the coefficient of $x^{\varphi(\beta)}$ in
$g\circ\Phi$.
Since $\varphi(\calB_i)=\calA_i$, this gives the natural identifications 
$\iota\colon\CC^{\calA_i}\xrightarrow{\,\sim\,}\CC^{\calB_i}$ and $\iota\colon\CC^\Adot\xrightarrow{\,\sim\,}\CC^\Bdot$
mentioned before the lemma.
Under this identification, we have $\iota(f)(\Phi(x))=f(x)$.

Since  $\MV(\Bdot)>1$, the branched cover $X_\Bdot\to\CC^\Bdot$ is nontrivial.
The identification $\iota\colon\CC^\Adot\to\CC^\Bdot$ extends to a commutative diagram
 \begin{equation}\label{Eq:Lacunary}
  \raisebox{-25pt}{\begin{picture}(97,58)(-4,0)
                          \put(35,52){\small$\iota\times\Phi$}  
    \put(0,45){$X_\Adot$} \put(24,49){\vector(1,0){47}} \put(73,45){$X_\Bdot$}
      \put(5,40){\vector(0,-1){27}}                    \put(78,40){\vector(0,-1){27}}
      \put(-4,25){\small$\pi$}                          \put(80,25){\small$\pi$}
                          \put(45,7){\small$\iota$}  
    \put(0,0){$\CC^\Adot$} \put(25,4){\vector(1,0){45}} \put(73, 0){$\CC^\Bdot$}
    \end{picture}}\ 
 \end{equation}
Here, $\iota\times\Phi$ is the restriction of the map
$\iota\times\Phi\colon \CC^\Adot\times(\CC^\times)^n \to \CC^\Bdot\times(\CC^\times)^n$ to $X_\Adot$.
The map $\iota\times\Phi\colon X_\Adot\to X_\Bdot$ is a map of branched covers with $\ker\Phi$ acting freely on the fibers.
If we restrict the diagram~\eqref{Eq:Lacunary} to the open subset $V$ of $\CC^\Bdot$ over which
$X_\Bdot\to\CC^\Bdot$ is a covering space, we obtain a composition of covering spaces with $\ker\Phi$ acting as deck
transformations on ${\pi^{-1}(V)}\subset X_\Adot$.
Thus $X_\Adot\to\CC^\Adot$ is decomposable.
\end{proof}

\subsubsection{Triangular support}\label{SSS:triangular}
This requires more discussion before we can state the analog of Lemma~\ref{L:Lacunary}.
Let us begin with an example when $n=3$.
Suppose that
\[
  \calA_1\ =\ \calA_2\ =\ \calA\ =\ 
  \begin{pmatrix} 0&1&1&1&2&2&2&3\\
                  0&0&1&2&0&1&2&1\\
                  0&1&2&3&2&3&4&4\end{pmatrix}
  \quad\mbox{and}\quad
 \calA_3\ =\
  \begin{pmatrix} 0&0&0&0&1&1\\
                  0&0&0&1&0&1\\
                  0&2&4&5&3&4\end{pmatrix}
\]
The span $\ZZ\calA$ of the first two supports is isomorphic to $\ZZ^2$, with 
$\varphi(a,b)^T\mapsto(a,b,a+b)^T$ an isomorphism $\varphi\colon\ZZ^2\xrightarrow{\,\sim\,}\ZZ\Adot$.
Set $\defcolor{\calB}:=\varphi^{-1}(\calA)$.
We display $\calA$, $\calA_3$, and $\calB$ in the horizontal plane together on the left below, and $\calB$ on the right. 
\[
  \begin{picture}(125,118)(-9,0)
    \put(0,0){\includegraphics{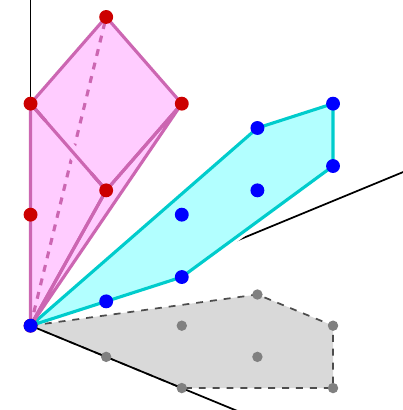}}
    \put(1,109){$z$} \put(32,2){$x$} \put(108,71){$y$}
    \put(-9,69){$\calA_3$}
    \put(60,17){$\calB$}
    \put(79,90){$\calA$}
  \end{picture}
    \qquad
  \begin{picture}(82,54)(-12,0)
    \put(0,0){\includegraphics{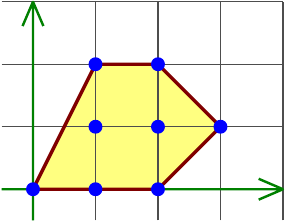}}
    \put(32,32){$\calB$}
  \end{picture}
\]
Consider the polynomial system $F=(f_1,f_2,f_3)\in\CC[x,y,z]$ with support $\Adot$,
\begin{align*}
  f_1 &=\ 1 + 2xz+3xyz^2+4xy^2z^3+5x^2z^2+6x^2yz^3+7x^2y^2z^4+8x^3yz^4\\
  f_2 &=\ 2 + 3xz+5xyz^2+7xy^2z^3+11x^2z^2+13x^2yz^3+17x^2y^2z^4+19x^3yz^4\\
  f_3 &=\ 1 + 3z^2 + 9 z^4 + 27 yz^5 + 81 xz^3 + 243 xyz^4\,.
\end{align*}
Let $\Phi\colon(\CC^\times)^3\to(\CC^\times)^2$ be given by $\Phi(x,y,z)=(xz,yz)=(u,v)$.
If 
\begin{align*}
  g_1 &=\ 1 + 2u+3uv+4uv^2+ 5u^2+ 6u^2v+ 7u^2v^2+ 8u^3v\\
  g_2 &=\ 2 + 3u+5uv+7uv^2+11u^2+13u^2v+17u^2v^2+19u^3v\,,
\end{align*}
then $f_i=g_i\circ\Phi$ for $i=1,2$.
To compute $\calV(F)$, we first may compute $\calV(g_1,g_2)$ which consists of eight points.
For each solution $(u_0,v_0)\in \calV(g_1,g_2)$, we may identify the fiber $\Phi^{-1}(u_0,v_0)$ with $\CC^\times$
by $z\mapsto (u_0z^{-1},v_0z^{-1},z)$.
Then the restriction of $f_3$ to this fiber is
\[
  1 + (3 + 81 u_0 + 243 u_0v_0)z^2 + (9 + 27 v_0 )z^4\,,
\]
which is a lacunary univariate polynomial with support $\{0,2,4\}$, and has four solutions (counted with multiplicity) when  
$v_0\neq -1/3$.

This example generalizes to all triangular systems.
Suppose  that $\Adot=(\calA_1,\dotsc,\calA_n)$ is triangular.
Let $\emptyset\neq I\subsetneq[n]$ be a proper subset witnessing the triangularity, so that 
$\rank(\ZZ\calA_I)=|I|$.
Set $\defcolor{J}:=[n]\smallsetminus I$.
Let
\[
  \defcolor{\ZZ^I}\ :=\ \QQ\calA_I\cap\ZZ^n\ =\
  \{ v\in\ZZ^n\mid \exists m\in\NN\mbox{ with }mv\in\ZZ\calA_I\}\,,
\]
be the saturation of $\ZZ\calA_I$, which is a free abelian group of rank $|I|$.
As it is saturated, $\defcolor{\ZZ_J}:=\ZZ^n/\ZZ^I$ is free abelian of rank $n-|I|=|J|$.

Applying $\Hom(\bullet,\CC^\times)$ to the short exact sequence
$\ZZ^I\hookrightarrow \ZZ^n\twoheadrightarrow\ZZ_J$ gives the short exact sequence of tori
(whose characters are $\ZZ_J$, $\ZZ^n$, and $\ZZ^I$) with indicated maps,
\begin{equation}
\label{Eq:Triangular}
  (\CC^\times)^{|J|}\simeq\defcolor{\TT_J}:=\Hom(\ZZ_J,\CC^\times)\
  \lhra\ (\CC^\times)^n\ \stackrel{\Phi}{\lthra}\
  \defcolor{\TT^I}:=\Hom(\ZZ^I,\CC^\times)\simeq(\CC^\times)^{|I|}\,.
\end{equation}
A  polynomial $f$ with support in $\ZZ^I$ determines polynomial functions on $(\CC^\times)^n$ and on $\TT^I$ with the first
the pullback of the second. 
Let $f$ be a polynomial on $(\CC^\times)^n$ with support $\calA\subset\ZZ^n$.
Then its restriction to a fiber $\Phi^{-1}(y_0)$ of $\Phi$ is a regular function $\overline{f}$ on the fiber, which
is a coset of $\TT_J$. 
Choosing an identification of $\TT_J\simeq\Phi^{-1}(y_0)$, we obtain a polynomial
\defcolor{$\overline{f}$} on  $\TT_J$ whose
support is the image \defcolor{$\overline{\calA}$} of $\calA$ in $\ZZ_J=\ZZ^n/\ZZ^I$.
This polynomial $\overline{f}$ depends upon the identification of the fiber with $\TT_J$.
Let $\defcolor{\overline{\calA_J}}$ be the image in $\ZZ_J$ of the collection $\calA_J$ of supports.
Then we have the product formula (see~\cite[Lem.~6]{ThSt} or~\cite[Thm.~1.10]{Esterov})
\begin{equation}\label{Eq:MVproduct}
  \MV(\Adot)\ =\ \MV(\calA_I)\cdot\MV(\overline{\calA_J})\,.
\end{equation}

Since $\Adot = \calA_I\sqcup\calA_J$, we have the identification $\CC^\Adot=\CC^{\calA_I}\oplus\CC^{\calA_J}$.
Suppose that $F(x)=0$ is a polynomial system with support $\Adot$.
Write $\defcolor{F_I}\in\CC^{\calA_I}$ for its
restriction to the indices in $I$, and the same for $F_J$.
We have the diagram
 \begin{equation}\label{Eq:Triangular_Decomposition}
  \raisebox{-25pt}{\begin{picture}(97,60)(-4,0)
                          \put(32,52){\small$p_I\times\Phi$}  
    \put(0,45){$X_\Adot$} \put(24,49){\vector(1,0){47}} \put(73,45){$X_{\calA_I}$}
      \put(5,40){\vector(0,-1){27}}                    \put(78,40){\vector(0,-1){27}}
      \put(-4,25){\small$\pi$}                          \put(80,25){\small$\pi$}
                          \put(44,8){\small$p_I$}  
    \put(0,0){$\CC^\Adot$} \put(25,4){\vector(1,0){45}} \put(73, 0){$\CC^{\calA_I}$}
    \end{picture}}\ .
 \end{equation}
Here, $p_I\times\Phi$ is the restriction of the map
$p_I\times\Phi\colon\CC^\Adot\times(\CC^\times)^n\to \CC^{\calA_I}\times\TT^I$ to $X_\Adot$.
 
Let $\defcolor{V_\Adot}\subset\CC^\Adot$ be the maximal Zariski open subset over which $X_\Adot$ is a covering space.
This is the set of polynomial systems $F(x)=0$ with support $\Adot$ which have exactly
$\MV(\Adot)$ solutions in $(\CC^\times)^n$.
Similarly, let $\defcolor{V_{\calA_I}}\subset\CC^{\calA_I}$ be the maximal Zariski open subset where $X_{\calA_I}\to\CC^{\calA_I}$ is a covering space.
We will show that under the projection $\CC^{\Adot}\to\CC^{\calA_I}$, the image of $V_\Adot$ is a subset of $V_{\calA_I}$.
Define $\defcolor{Y_{\Adot}}\to V_\Adot$ to be the restriction of $X_\Adot\to\CC^{\Adot}$ to the Zariski open set $V_\Adot$.
Also define $\defcolor{Y_{\calA_I}}\to V_\Adot$ to be the pullback of $X_{\calA_I}\to\CC^{\calA_I}$ along the map
$V_\Adot\to V_{\calA_I}$.
Write $\Phi\colon Y_\Adot\to Y_{\calA_I}$ for the map induced by $\Phi$.

\begin{lemma}\label{L:Triangular}
 Suppose that $\Adot$ is a triangular set of supports in $\ZZ^n$ witnessed by $I\subsetneq[n]$.
 Then $Y_\Adot\to Y_{\calA_I}\to V_{\Adot}$ a composition of covering spaces.
 If $1<\MV(\calA_I)<\MV(\Adot)$, then this decomposition is nontrivial, so that $X_\Adot\to\CC^\Adot$ is decomposable.

 Furthermore, each fiber of the map $Y_\Adot\to Y_{\calA_I}$ may be identified with the set of solutions to a polynomial
 system with support $\overline{\calA_J}$. 
\end{lemma}

\begin{proof}
  Let $F\in V_{\Adot}$.
  Then its number of solutions is $\#\calV(F)=\MV(\Adot)$.
  If $x\in\calV(F)$, then $\Phi(x)\in\TT^I$ is a solution to $f_i=0$ for $i\in I$.
  Thus $\Phi(\calV(F))\subset\calV(F_I)$, the latter being the solutions to $F_I(x)=0$ on $\TT^I$.
  For any $y\in\calV(F_I)$, if we choose an identification $\TT_J \simeq\Phi^{-1}(y)$ of the fiber, 
  then the restriction of $F$ to $\Phi^{-1}(y)$ is the system $\defcolor{\overline{F_J}}=\{\overline{f_j}\mid j\in J\}$.
  By the Bernstein-Kushnirenko Theorem, this has at most $\MV(\overline{\calA_J})$ solutions.
  By the product formula~\eqref{Eq:MVproduct} and our assumption on $\#\calV(F)$, we conclude that
  the system $F_I(x)=0$ has $\MV(\calA_I)$ solutions, and for each $y\in\calV(F_I)$, the system $\overline{F_J}$ has 
  $\MV(\overline{\calA_J})$ solutions.

  In particular, this implies that the image of $V_{\Adot}$ in $\CC^{\calA_I}$ is a subset of $V_{\calA_I}$.
  As  $V_{\Adot}$ is open and dense in $\CC^\Adot$, its image contains an open dense subset.
  This proves the assertion that  $Y_\Adot\to Y_{\calA_I}\to V_{\Adot}$ is a decomposition of covering spaces.
  We have already shown that each fiber  of the map $Y_\Adot\to Y_{\calA_I}$ is a polynomial system with support
  $\overline{\calA_J}$ with exactly $\MV(\overline{\calA_J})$ solutions.
  Thus when $1<\MV(\calA_I)<\MV(\Adot)$, we have $\MV(\overline{\calA_J})>1$, which shows that this
  decomposition is nontrivial.
\end{proof}

\subsection{Computing the Decompositions}\label{S:SNF}
We show how to compute the decompositions of $X_\Adot\to\CC^{\Adot}$ from 
Section~\ref{SS:Decompositions} when $\Adot$ is either lacunary or strictly triangular.

Let $\calA=\{0,\alpha_1,\dotsc,\alpha_m\}\subset\ZZ^n$ be a collection of integer vectors.
The subgroup $\ZZ\calA\subset\ZZ^n$ that it generates is the image of a $\ZZ$-linear map $\ZZ^m\to\ZZ^n$ and is
represented by a $n\times m$ integer matrix \defcolor{$A$} whose columns are the vectors $a_i$.
Suppose that $\ZZ\calA$ has rank $k$.
A \demph{Smith normal form} of $A$ is a factorization into integer matrices
 \begin{equation}\label{Eq:SNF}
  A\ =\ PDQ\,,
 \end{equation}
where $P\in\GL_n(\ZZ)$ and $Q\in\GL_m(\ZZ)$ are invertible,
and $D$ is the rectangular matrix whose only nonzero entries are $d_1,\dotsc,d_k$ along the diagonal of its principal 
$k\times k$ submatrix.
These are the  \demph{invariant factors} of $A$ and they satisfy $d_1|d_2|d_3|\dotsb|d_k$.

The subgroup $\ZZ\calA\subset\ZZ^n$ has a basis given by the columns of the matrix $PD$.
If we apply the coordinate change $P^{-1}$ to $\ZZ^n$, then $\ZZ\calA$ becomes the subset of the coordinate space
$\ZZ^k\oplus\bzero^{n-k}$ given by $d_1\ZZ\oplus d_2\ZZ\oplus\dotsb\oplus d_k\ZZ\oplus\bzero^{n-k}$.

Let us consider a Smith normal form~\eqref{Eq:SNF} when $\calA$ is the collection of vectors in $\Adot$ and
$\MV(\Adot)>0$. 
Then $d_n> 0$ as $\ZZ\Adot$ has rank $n$, and $\Adot$ is lacunary when $d_n> 1$.
In this case, an identification $\varphi\colon\ZZ^n\xrightarrow{\,\sim\,}\ZZ\calA$ is given by $PD_n$, where
\defcolor{$D_n$} is the principal $n\times n$ submatrix of $D$.
Recall from~\S~\ref{SS:CoordinateChanges} that the 
corresponding surjection $\varphi^*=\Phi\colon(\CC^\times)^n\to(\CC^\times)^n$ has kernel
$\Hom(\ZZ^n/\ZZ\Adot,\CC^\times)$.
Let $\defcolor{\psi}:=P^{-1}$.
Then $\psi\circ\varphi=D_n$, so that if we set $\defcolor{\Psi}:=\psi^*$, then
$\Phi\circ\Psi\colon(\CC^\times)^n\to(\CC^\times)^n$  is diagonal, 
 \begin{equation}\label{Eq:diagonal}
  \Phi\circ\Psi(x_1,\dotsc,x_n)\ =\ (x_1^{d_1},\dotsc,x_n^{d_n})\,.
 \end{equation}
Let $y=(y_1,\dotsc,y_n)\in(\CC^\times)^n$.
If we set $\defcolor{\rho_i}:=|y_i|$ and $\defcolor{\zeta_i}:=\arg(y_i)$ so that $y_i=\rho_i e^{\sqrt{-1}\zeta_i}$, then 
$(\Phi\circ\Psi)^{-1}(y)$ is the set
 \begin{equation}\label{Eq:PhiInv}
    \left\{ \left(\rho_1^{1/d_1}e^{\sqrt{-1}\theta_1}\,,\,\dotsc\,,\,\rho_n^{1/d_n}e^{\sqrt{-1}\theta_n}\right)
   \,\middle\vert\,  \theta_i = \tfrac{\zeta_i{+}2\pi j}{d_i}\mbox{ for } j=0,\dotsc,d_i{-}1\right\}.\vspace{3pt}
 \end{equation}
This ends the discussion of lacunary sparse polynomial systems.


Suppose that $\Adot$ is triangular, and let us use the notation of \S~\ref{SSS:triangular}.
We suppose that $I=[k]=\{1,\dotsc,k\}$ and $J=\{k{+}1,\dotsc,n\}$.
Given a polynomial $f$ on $(\CC^\times)^n$, its restriction $\overline{f}$ to a fiber of $\Phi\colon(\CC^\times)^n\to\TT^I$
is a regular function on the fiber, which is isomorphic to $\TT_J$.
To represent $\overline{f}$ as a polynomial on $\TT_J$ depends on the choice of a point in that fiber. 
Indeed, suppose that $f=\sum_{\alpha\in\calA}c_\alpha x^\alpha$.
Let $y\in\TT^I$ and $y_0\in\Phi^{-1}(y)$ be a point in
the fiber above $y$, so that $\TT_J\ni z\mapsto y_0z\in\Phi^{-1}(y)$ parameterizes $\Phi^{-1}(y)$.
If we write \defcolor{$\overline{\alpha}$} for the image of $\alpha\in\ZZ^n$ in $\ZZ_J=\ZZ^n/\ZZ^I$, then 
 \begin{equation}\label{Eq:substituteFibre}
   \overline{f}(z)\ =\ \sum_{\alpha\in\calA} c_\alpha (y_0z)^\alpha\
    =\ \sum_{\beta\in\overline{\calA}} z^\beta \ 
   \biggl(\,\sum_{\alpha\in\calA\ \mbox{\scriptsize with}\ \overline{\alpha}=\beta} c_\alpha y_0^\alpha\biggr)\,.
 \end{equation}
A uniform choice of a point in each fiber is given by a splitting
$\TT^I\hookrightarrow(\CC^\times)^n$ of the map $\Phi\colon(\CC^\times)^n\twoheadrightarrow\TT^I$.
This gives an identification $(\CC^\times)^n=\TT^I\times\TT_J$.
Then points $y\in\TT^I$ are canonical representatives of cosets of $\TT_J$.
As $k=|I|$, we may further fix isomorphisms $\TT^I\simeq(\CC^\times)^k$ giving $\ZZ^I\simeq\ZZ^k$ and
$\TT_J\simeq(\CC^\times)^{n-k}$ giving $\ZZ_J\simeq\ZZ^{n-k}$.


Suppose now that $\calA=\calA_I$, and we compute a decomposition~\eqref{Eq:SNF}.
Since $\ZZ\calA_I$ has rank $k$, the diagonal matrix $D$ has $k$ nonzero invariant factors.
The saturation $L$ of  $\ZZ\calA_I$ is the image of $P\IK_k$, where $\IK_k$ is the $n\times n$
matrix whose only nonzero entries are in its principal $k\times k$ submatrix, which forms an identity matrix.
Then $\varphi=P\IK_k$ and $\Phi=\varphi^*$.
Applying the coordinate change $\defcolor{\psi}:=P^{-1}$ to $\ZZ^n$ identifies this saturation as the coordinate plane 
$\ZZ^k\oplus\bzero^{n-k}$ and the free abelian group $\ZZ\calA_I$ as $d_1\ZZ\oplus d_2\ZZ\oplus\dotsb\oplus d_k\ZZ\oplus\bzero^{n-k}$.
As in Section~\ref{SSS:triangular}, this identifies $\ZZ/L$ with the complementary coordinate plane,
$\bzero^k\oplus\ZZ^{n-k}$.
Setting $\defcolor{\Psi}:=\psi^*$ , the composition $\Phi\circ\Psi$ is the projection to the first $k$ coordinates,
 \begin{equation}\label{Eq:projection}
  \Phi\circ\Psi\;\colon\; (\CC^\times)^n\ \lthra\ (\CC^\times)^k
 \end{equation}
and we identify $\TT_J=1^k\times (\CC^\times)^{n-k}$ and 
$\TT^I=(\CC^\times)^k\times 1^{n-k}$.

\section{Algorithms for Solving Sparse Decomposable Systems}\label{S:Algorithm}
We describe algorithms for solving sparse decomposable systems, and suggest an application of these algorithms to computing
a start system to solve general systems (not necessarily decomposable) of sparse polynomials.
They are based on numerical homotopy continuation~\cite{Morgan}.
By ``solve a system of polynomials'',
 we mean compute numerical approximations to the complex isolated solutions which may then be refined using 
Newton iterations.
The expected numbers of isolated solutions to the systems we consider are mixed volumes as explained in the
Bernstein-Kushnirenko Theorem.
In principle, as the systems are square and we know the expected number of isolated solutions, Smale's $\alpha$-theory~\cite{Smale}
enables approximations to solutions to be certified as approximate solutions in that Newton iterations converge
quadratically to solutions, as explained in~\cite{HS12}.

Let $\Adot$ be a collection of supports with $\MV(\Adot)>0$, so that $\pi\colon X_\Adot\to\CC^\Adot$ is a branched cover
and let $F \in \CC^\Adot$.
A \demph{start system} for $\Adot$ is a pair $(G,\calV(G))$ where $G\in\CC^\Adot$ and $\calV(G)$ consists of 
$\MV(\Adot)$ distinct points.
The convex combination of systems 
 \begin{equation}\label{Eq:parameter}
  \defcolor{H(t)}\ =\ tF + (1-t)G\qquad\mbox{for }t\ \in\ \CC\,,
 \end{equation}
is a straight-line homotopy.
Then $\calV(H(t))\subset(\CC^\times)^\Adot\times\CC$ 
defines a curve $C$.
Forgetting the $x$-coordinates gives a dominant map $C\to\CC$.
Restricting to points above $t\in[0,1]$ gives a family of arcs on $C$.
Starting with the points of $\calV(G)$ at $t=0$, path-tracking along these arcs using $H(t)$ will give
isolated solutions to $\calV(F)$ at $t=1$ when $F$ is a regular value of $\pi$.
This is an instance of a \demph{(parameter) homotopy}~\cite{LSY89,MS89}.
If $F$ is not a regular value but $\calV(F)$ is still finite, then $\calV(F)$ may be computed using
endgames~\cite{BHS11,HV98}. Problems of numerically tracking solutions are treated in \cite{Morgan,SW05}.

\subsection{Solving decomposable sparse systems}\label{ss:sdss}
We describe algorithms that use Esterov's conditions to solve a decomposable sparse polynomial system.
In each, we let \defcolor{$\texttt{SOLVE}$} be an arbitrary algorithm for solving a polynomial system.
We assume that it is known that the system $F(x)=0$ to be solved is 
general given its support $\Adot$ in that it has $\MV(\Adot)$  solutions in $(\CC^\times)^n$.
If not, then one instead solves a polynomial system with support $A_\bullet$ whose coefficients are random complex numbers.
With probability one, this system is generic and one may use a homotopy together with endgames 
 to compute the isolated solutions to $\calV(F)$.

Our main algorithm (Algorithm~\ref{alg:SDS})  takes a sparse system and checks Esterov's criteria for decomposability.  
If the system is decomposable, the algorithm calls Algorithm~\ref{alg:Lacunary} (if lacunary) or
Algorithm~\ref{alg:Triangular} (if triangular), and in each of these algorithms calls to the solver $\texttt{SOLVE}$ are assumed to
be recursive calls back to Algorithm~\ref{alg:SDS}.
If the polynomial system is indecomposable, then Algorithm~\ref{alg:SDS}
calls a black box solver $\defcolor{\texttt{BLACKBOX}}$.

Recall from Section~\ref{SS:CoordinateChanges}  the relation between the linear map $\varphi=\Phi^*$ and the group
homomorphism  $\Phi=\varphi^*$.  
Furthermore, recall the identification $\iota$ in \eqref{Eq:Lacunary}.

\begin{algorithm}[SolveLacunary]\ 
\label{alg:Lacunary}

{\bf Input:}  A general polynomial system $F(x)=0$ whose support $\Adot$ is  lacunary.

{\bf Output:} All isolated solutions $\calV(F)\subset (\CC^\times)^n$.

{\bf Do:}
\begin{enumerate}
\item  Compute a Smith normal form \eqref{Eq:SNF} of $\Adot$, giving  $\varphi=PD_n$, $\Phi=\varphi^*$, $\psi = P^{-1}$, and $\Psi=\psi^*$,
  so that $\Phi\circ\Psi$ is diagonal~\eqref{Eq:diagonal}.
\item Use $\texttt{SOLVE}$ to compute isolated solutions of $\iota(F)(x)=0$ in $(\CC^\times)^n$.

\item Using the formula~\eqref{Eq:PhiInv} to compute $(\Phi\circ\Psi)^{-1}(y)$ for $y\in\calV(\iota(F))$, return
  \[
    \{ \Psi(w) \mid w\in \bigcup_{y \in\calV(\iota(F))} (\Phi\circ\Psi)^{-1}(y)\}\,.
  \]
\end{enumerate}
\end{algorithm}
\begin{proof}[Proof of Correctness.]
  By Lemma~\ref{L:Lacunary}, $\calV(F)=\Phi^{-1}(\calV(\iota(F)))$.
  We apply $\Psi$ to points of $(\Phi\circ\Psi)^{-1}(y)$ for $y\in\calV(\iota(F))$ to obtain
  points of $\calV(F)$ in their original coordinates.
\end{proof}

Recall the notation ${F_I}\in\CC^{\calA_I}$ used in \eqref{Eq:Triangular_Decomposition}.

\begin{algorithm}[SolveTriangular]\ 
\label{alg:Triangular}

{\bf Input:}  A general polynomial system $F(x)=0$ whose support $\Adot$ is triangular, witnessed by $0<k<n$ such that 
 $\rank(\ZZ\calA_{[k]})=k$.

{\bf Output:} All isolated solutions to $\calV(F)\subset(\CC^\times)^n$.

{\bf Do:}
\begin{enumerate}
\item  Compute a Smith normal form \eqref{Eq:SNF} of $\calA_{[k]}$, giving $\varphi=P\IK_k$, $\Phi=\varphi^*$, $\psi = P^{-1}$, and
  $\Psi=\psi^*$, so that  $\Phi\circ\Psi$ is the projection~\eqref{Eq:projection}.
  
\item Use $\texttt{SOLVE}$ to compute isolated solutions to $F_{[k]}(x)=0$ in $(\CC^\times)^k$.

\item Choose $y_0\in\calV(F_{[k]})$.
  Use $\texttt{SOLVE}$ to compute the points of the fiber $(\Phi\circ\Psi)^{-1}(y_0)$ in $Y_{\Adot}$,
  which are $\calV(\overline{F_J})\subset\{y_0\}\times(\CC^\times)^{n-k}$, where $\overline{F_J}(x)=0$ has support
  $\overline{\calA_J}$ and $J:=[n]\smallsetminus[k]$. 
 
\item For each $y\in\calV(F_{[k]})$ use a homotopy~\eqref{Eq:parameter} with start system
  $\calV(\overline{F_J})$ to compute $(\Phi\circ\Psi)^{-1}(y)$ and return
  \[
    \{ \Psi(w) \mid w\in \bigcup_{y \in\calV(F_{[k]})} (\Phi\circ\Psi)^{-1}(y)\}\,.
  \]
\end{enumerate}
\end{algorithm}
\begin{proof}[Proof of Correctness.]
  By Lemma~\ref{L:Triangular}, every solution $x\in\calV(F)$ lies over a solution $y=\Phi(x)$ to $F_{[k]}(x)=0$ in
  $(\CC^\times)^k$.
  As explained in Section~\ref{S:SNF}, the map $\Phi\circ\Psi$ is a coordinate projection and
  $(\Phi\circ\Psi)^{-1}(y)=\calV(\overline{F_J})$.
  Here, $\overline{F_J}=(\overline{f_{k+1}},\dotsc,\overline{f_n})$ where $\overline{f_j}$ has support $\overline{\calA_j}$
  and is computed using~\eqref{Eq:substituteFibre}.
  We apply $\Psi$ to convert these points to the original coordinates.  
\end{proof}

The previous two algorithms handle decomposable systems that are lacunary or triangular. 
We now state our main algorithm and later illustrate it in detail for a decomposable support $\Adot$ in Example~\ref{Ex:computation1}.  
We remark that our methods can be used as a preprocessing step for a black box solver.

\begin{algorithm}[SolveDecomposable]\ 
\label{alg:SDS}

{\bf Input:} A generic polynomial system $F(x)=0$ with support $\Adot$.

{\bf Output:} All isolated solutions to $\calV(F)\subset (\CC^\times)^n$.

{\bf Do:}
\begin{enumerate}
\item Compute a Smith normal form $PDQ$~\eqref{Eq:SNF} of $\Adot$.\newline
  {\bf if} $d_n>1$, then {\bf return} SolveLacunary$(F)$.\newline
  {\bf if} $d_n=1$, then

   {\bf for all} $\emptyset\neq I\subsetneq [n]$
    compute a Smith normal form  $PD_IQ$~\eqref{Eq:SNF} of $\calA_I$.

 \item  \quad  \quad {\bf if} $\rank(D_I)=|I|$ for some $I$, reorder so that $I=[k]$ and 

   \quad  \quad  \quad  {\bf return} SolveTriangular$(F,k)$.
    
\item  \quad  \quad {\bf else} neither of Esterov's conditions hold and {\bf return} $\texttt{BLACKBOX}(F)$.
\end{enumerate}
\end{algorithm}
\begin{proof}[Proof of Correctness.]
  First note that if the algorithm halts, then it returns the isolated solutions of $F(x)=0$.
  Halting is clear in Case (3), but the other cases involve recursive calls back to  Algorithm~\ref{alg:SDS}.
  In Case (1), SolveLacunary will call Algorithm~\ref{alg:SDS} on a system $\iota(F)(x)=0$ whose mixed volume is
  less than $\MV(\Adot)$.
  In Case (2), SolveTriangular  will call Algorithm~\ref{alg:SDS} on systems $F_{[k]}(x)=0$ and $\overline{F_J}(x)=0$, each involving
  fewer variables than $F(x)=0$.
  Thus, in each recursive call back to  Algorithm~\ref{alg:SDS}, either the mixed volume or the number of variables
  decreases, which proves that the algorithm halts.
\end{proof}

\subsection{Start Systems}

The B\'ezout homotopy~\cite{Garcia79} is a well-known  homotopy for solving a system $F=(f_1,\dotsc,f_n)=0$ 
where each $f_i$ is a general polynomial of degree $d_i$.
In it, the start system is $G=(x_1^{d_1}-z_1,\dotsc,x_n^{d_n}-z_n)$, and $\calV(G)=\Phi^{-1}(z)$, where $\Phi$ is the
diagonal map~\eqref{Eq:diagonal} and $\Phi^{-1}(z)$ is determined by inspection from~\eqref{Eq:PhiInv}. This start system is a highly decomposable sparse polynomial system consisting of supports which are subsets of the original support of $F(x)=0$, but have the same mixed volume.
We propose a generalization, in which Algorithm~\ref{alg:SDS} is used to compute a start system.

\begin{example}\label{Ex:start}
  Suppose that we have supports $\calA_1=\calA_2=\calA$, which are given by the columns of the matrix
  $(\begin{smallmatrix}0&0&1&1&2&3&3&3&4&5&5&6\\0&2&0&1&3&0&1&4&2&3&4&4\end{smallmatrix})$.
  Then $\MV(\calA_1,\calA_2)=\vol(\conv(\calA))=30$.
  Let $\calB_1=\calB_2=(\begin{smallmatrix}0&0&3&3&6\\0&2&0&4&4\end{smallmatrix})$\vspace{1pt} be the set of vertices
  of $\conv(\calA)$.
  Given a general $F\in\CC^\Adot$, let $G\in\CC^\Bdot\subset\CC^\Adot$ be obtained from $F$ by restriction to the
  monomials in $\calB$.
  (That is, we set coefficients of monomials $x^\alpha$ in $F$ to zero if $\alpha\not\in\calB$.)
  Then $\Bdot$ is lacunary with the map $\Phi(x_1,x_2)=(x_1^3,x_2^2)$, and $\iota(G)$ has five solutions, say $(z_i,w_i)\in \mathbb{C}^2$ for $i=1,2,..,5$.
  The left, center, and right figures give the support of the polynomials
appearing in the systems $F=0$, $G=0$, and $\iota(G)=0$ respectively.
(The blue and red dots correspond to monomials with a nonzero coefficient.)
    \[
   \includegraphics{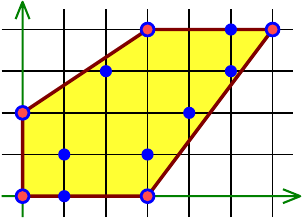}\quad\quad
   \includegraphics{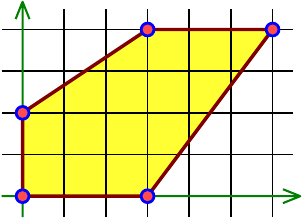}\quad\quad
   \includegraphics{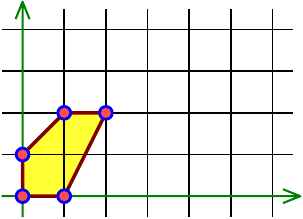} 
  \]

  We may use Algorithm~\ref{alg:SDS} (more specifically, Algorithm~\ref{alg:Lacunary}) to compute $\calV(G)$. 
 \newcommand{\thirdrootz}{\sqrt[\leftroot{-1}\uproot{2}\scriptstyle 3]{z_i}}
\newcommand{\sqrtw}{\sqrt{w_i}}
The thirty solutions to $G=0$ are
\[
(\thirdrootz,\sqrtw),\;(\eta \thirdrootz,\sqrtw),\;( \eta^2 \thirdrootz,\sqrtw),\;
(\thirdrootz,-\sqrtw),\;      ( \eta \thirdrootz,-\sqrtw),\;(\eta^2 \thirdrootz,-\sqrtw)
\] 
for $i=1,2,..,5$ and where $\eta$ is a primitive third root of unity. 
We compute the isolated solutions of $F(x)=0$ using the   homotopy
  \begin{equation}\label{Eq:SLH}
    H(t)\ =\ t F\ +\ (1-t)G
  \end{equation}
  with start system $G=H(0)$ and tracking from the thirty solutions $\calV(G)$ at $t=0$.
  \hfill$\diamond$ 

\end{example}


Example~\ref{Ex:start} motivates our final algorithm.
For a finite set $\calA\subset\ZZ^n$, let $\defcolor{v(\calA)}\subset\calA$ be the subset of vertices of $\conv(\calA)$. 
For a collection $\Adot=(\calA_1,\dotsc,\calA_n)$ of supports, let $\defcolor{v(\Adot)}:=(v(\calA_1),\dotsc,v(\calA_n))$.
Note that if $G \in \CC^{v(\Adot)}$ is a regular value of the branched cover
$\pi|_{X_{v(\Adot)}}\colon X_{v(\Adot)} \to \CC^{v(\Adot)}$ then $G$ is also a regular value of
$\pi\colon X_\Adot \to \CC^\Adot$.
As such, $G(x)=0$ may be taken as a start system for the   homotopy~\eqref{Eq:SLH} and may be used to compute $\calV(F)$
for any $F\in\CC^{\Adot}$ with $\calV(F)$ finite.
The benefit of this approach, as seen in Example~\ref{Ex:start}, is that $\pi|_{X_{v(\Adot)}}$ is more likely than $\pi$ to
be lacunary.

\begin{algorithm}[Decomposable Start System]\ 
\label{alg:SDSS}

{\bf Input:} A set $\Adot$ of supports.

{\bf Output:} A start system $(G,\calV(G))$ for $\Adot$.

{\bf Do:}
\begin{enumerate}
  \item Choose a general system $G\in\CC^{v(\Adot)}$.
  \item Compute $\calV(G)$ using Algorithm~\ref{alg:SDS}.
  \item {\bf return} the pair $(G,\calV(G))$.
\end{enumerate}

\end{algorithm}
\begin{proof}[Proof of Correctness.]
  As $G\in\CC^{v(\Adot)}$ is  general, it has $\MV(v(\Adot))$ solutions.
  Since for each $i$, $\conv(\calA_i)=\conv(v(\calA_i))$, we have $\MV(v(\Adot))=\MV(\Adot)$.
  Finally, $\CC^{v(\Adot)}$ is the subspace of $\CC^{\Adot}$ where the coefficients of nonextreme monomials in each
  polynomial are zero.
  Thus $G\in\CC^{\Adot}$, which shows that $(G,\calV(G))$ is a start system for $\Adot$.
\end{proof}

\begin{remark}\label{R:more}
  The B\'ezout homotopy motivated the idea behind Algorithm~\ref{alg:SDSS}.
  However, if we apply Algorithm~\ref{alg:SDSS} to the system of supports $\Adot$, where $\calA_i$ is all monomials of
  degree at most $d_i$, then we will not get the start system for the B\'ezout homotopy.
  For example, when $n=2$, $d_1=2$, and $d_2=3$, the supports are  as shown.
  \[
   \begin{picture}(39,51)(0,-12)
      \put(0,0){\includegraphics{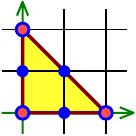}}
      \put(12,-11){$\calA_1$}
   \end{picture}
    \qquad
   \begin{picture}(51,63)(0,-12)
      \put(0,0){\includegraphics{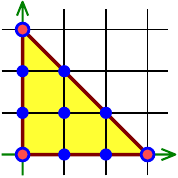}}
      \put(18,-11){$\calA_2$}
   \end{picture}\qquad
   \begin{picture}(39,51)(0,-12)
      \put(0,0){\includegraphics{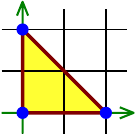}}
      \put(6,-11){$v(\calA_1)$}
   \end{picture}
    \qquad
   \begin{picture}(51,63)(0,-12)
      \put(0,0){\includegraphics{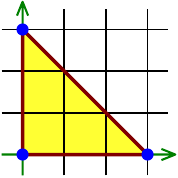}}
      \put(11,-11){$v(\calA_2)$}
   \end{picture}\qquad
   \begin{picture}(39,30)(0,-12)
      \put(0,0){\includegraphics{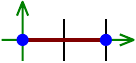}}
      \put(13,-11){$\calB_1$}
   \end{picture}
    \qquad
   \begin{picture}(18,63)(0,-12)
      \put(0,0){\includegraphics{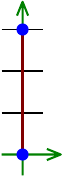}}
      \put(3,-11){$\calB_2$}
   \end{picture}
 \]
 Here, $\calB_1$ and $\calB_2$ are the supports of the start system for the B\'ezout homotopy.

 We leave open the challenge of finding a simple, general method to replace each set $\calA_i$ by a subset
 $\calB_i$ of $v(\calA_i)$, so that $\MV(\Adot)=\MV(\Bdot)$ and $\pi\colon X_{\Bdot} \to \CC^\Bdot$ is decomposable.
 A possible first step would be to refine the methods of~\cite{Chen}.
 This may lead to a simpler start system for a   homotopy to solve general systems with support $\Adot$.
\hfill$\diamond$
\end{remark}

\section{A computational experiment}\label{S:computations}
We explored the computational cost of using Algorithm~\ref{alg:SDS} to solve sparse decomposable systems, comparing
timings to PHCPack~\cite{V99,PHC_M2} on a family of related systems.

Let $\calA_1=(\begin{smallmatrix}0&1&2&0&1\\0&0&0&1&1\end{smallmatrix})$, 
$\calA_2=(\begin{smallmatrix}1&0&1&2&1\\0&1&1&1&2\end{smallmatrix})$, 
$\calB_1=(\begin{smallmatrix}0&2&0&2\\0&0&1&3\end{smallmatrix})$, and 
$\calB_2=(\begin{smallmatrix}0&1&2&0&2&0\\0&0&0&1&1&2\end{smallmatrix})$.
We display these supports and their convex hulls below.
\[
  \calA_1\ \raisebox{-11pt}{\includegraphics{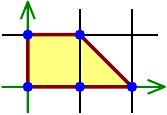}}\qquad
  \calA_2\ \raisebox{-18.5pt}{\includegraphics{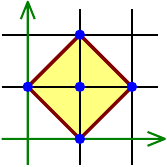}}\qquad\qquad
  \calB_1\ \raisebox{-26pt}{\includegraphics{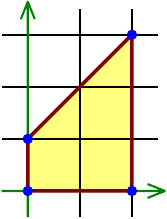}}\qquad
  \calB_2\ \raisebox{-18.5pt}{\includegraphics{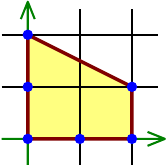}}
\]
Let $\defcolor{\calC}:=\{0,1\}^5$ be the vertices of the five-dimensional cube.
We construct sparse decomposable systems from $\defcolor{\Adot}:=(\calA_1,\calA_2)$,
$\defcolor{\Bdot}:=(\calB_1,\calB_2)$, and $\calC$ as follows.

Choose two injections $\defcolor{\imath},\defcolor{\jmath}\colon\ZZ^2\to\ZZ^5$ such that
$\imath(\ZZ^2)\cap\jmath(\ZZ^2)=\{0\}$.
For example, choose four linearly independent vectors $\imath_1,\imath_2,\jmath_1,\jmath_2\in\ZZ^5$, and define
$\imath(a,b)=a\imath_1+b\imath_2$, and the same for $\jmath$.
Let us set
\[
  \defcolor{\calA(\imath,\jmath)}\ :=\
  \bigl(\imath(\calA_1)\,,\,\imath(\calA_2)\,,\ \jmath(\calB_1)\,,\,\jmath(\calB_2)\,,\ \calC\bigr)\,.
\]

%
%
%
\begin{example}\label{Ex:computation1}
We now illustrate Algorithm~\ref{alg:SDS} in detail on $\calA(\imath,\jmath)$ by considering the case when
$\imath_1,\imath_2,\jmath_1,\jmath_2$ are the first four standard unit vectors $e_1,\ldots,e_4$.  
  Suppose $F = (f_1,f_2,g_1,g_2,h)=0$ is a system of polynomials $\mathbb{C}[x_1,x_2,y_1,y_2,z]$ with support
  $\mathcal A(e_1,e_2,e_3,e_4)$.  
  We use superscripts to distinguish different calls of the same algorithm.
  When \texttt{SolveDecomposable}$^{(1)}(F)$ is called, it first checks if $F$ is lacunary (it is not as $\ZZ\calC=\ZZ^5$),
  and then recognizes that $F$ is triangular witnessed by $(f_1,f_2)$.
  As such, it calls \texttt{SolveTriangular}$^{(1)}(F,2)$ which computes the $\MV(\Adot)=5$ solutions $p_1,\ldots,p_5$ to
  $\calV(f_1,f_2)$ with PHCPack, our choice of \texttt{BLACKBOX}.

  As its penultimate task, \texttt{SolveTriangular}$^{(1)}$ computes a fiber of the first solution $p_1$ by
  performing the substitution $(x_1,x_2)=p_1$ in $g_1,g_2$ and $h$, and recursively calls
  \texttt{SolveDecomposable}$^{(2)}$ on the system $(g_1(p_1,y,z),g_2(p_1,y,z),h(p_1,y,z))\in\CC[y_1,y_2,z]$.
  This system is recognized to be triangular witnessed by $(g_1,g_2)$ and \texttt{SolveTriangular}$^{(2)}(g_1,g_2)$
  computes the $\MV(\Bdot)=10$ solutions $q_1,\ldots,q_{10}$ using PHCPack.
  Next, \texttt{SolveTriangular}$^{(2)}$ computes a fiber above $q_1$ by performing  the substitution
  $y=(y_1,y_2)=q_1$ in $h(p_1,y,z)$ producing the univariate polynomial $h(p_1,q_1,z)$ of degree $1$ which
  has solution $(p_1,q_1,z_1)$.
  Finally, \texttt{SolveTriangular}$^{(2)}$ performs a   homotopy from $q_1$ to $q_i$ to populate the fibers above
  each $q_i$.
  Thus \texttt{SolveTriangular}$^{(1)}$ populates the fiber above $p_1$ consisting of $10 \cdot 1=10$ solutions.
  As its final step, \texttt{SolveTriangular}$^{(1)}$  uses homotopies  to take $p_1$ to $p_i$ to populate all
  fibers producing all $5 \cdot 10 = 50$ solutions, $\calV(F)$.\hfill$\diamond$
\end{example}

The overhead of this algorithm includes computing Smith normal forms and the search for subsets witnessing
triangularity.
Additionally, it often requires more path-tracking than a direct use of PHCPack.
Nonetheless, the overhead seems to be nominal, and compared to the paths tracked in PHCPack, the paths tracked in our
algorithm either involve fewer variables or  polynomials of smaller degree. 

For example, in Example \ref{Ex:computation1}, our algorithm called PHCPack to solve two sparse polynomial systems with $5$
and $10$ solutions respectively.
A   homotopy was called $10-1=9$ times on a system with $1$ solution, then a different  homotopy was
called $5-1=4$ times on a system with $10$ solutions.
In total, $5+10+9+40=64$ individual paths were tracked.
In contrast, a direct use of PHCPack involves tracking exactly $\MV(\mathcal A(e_1,e_2,e_3,e_4)) = 50$ paths, albeit
in a higher dimensional space.  

For more general 
$\imath$ and $\jmath$, the recursive structure of our computation is similar to
Example~\ref{Ex:computation1}. 
Some notable differences include 
\begin{enumerate}
\item $\imath(\Adot)$ or $\jmath(\Bdot)$ may be lacunary which induces further decompositions. 
\item Monomial changes must be computed as $\imath(\Adot)$ or $\jmath(\Bdot)$ could involve all variables.
\item For most $\imath,\jmath$ the univariate polynomial obtained from $h$ has degree $5$ and is solved by computing
  eigenvalues of its companion matrix.
\end{enumerate}
For example,
 if we choose $e_1-e_2, e_2-e_3,e_3-e_4, e_4-e_5$ for $\imath_1,\imath_2,\jmath_1,\jmath_2$, then again, no
system encountered in the algorithm is lacunary, but the univariate polynomial obtained from $h$ has support
$\{0,1,2,3,4,5\}$, so that $\MV(\calA(\imath,\jmath))=250$.

In our computational experiment, we produced $10962$ instances of $\mathcal A(\imath,\jmath)$ and solved each instance using our implementation of Algorithm
\ref{alg:SDS} as well as with PHCPack.
Due to ill-conditioning 
and heuristic choices of tolerances, some computations failed to produce all solutions.
Such occurrences are not included in the data displayed below.

We give a scatter plot of the elapsed timings in Figure \ref{fig:scatter} with respect to the mixed volume of the
system. Figure \ref{fig:boxplots} displays box plots of the timings of each algorithm grouped by sizes of mixed
volumes.
The boxes range from the first quartile $q_1$ to the third quartile $q_3$ of the group data with whiskers
extending to the smallest and largest data points which are not outliers.
Outliers are the data points which are smaller than $q_1-1.5I$ or larger than $q_3+1.5I$ where $I$ is the length of the
interquartile range $(q_1,q_3)$. 

\begin{figure}[htb!]
	\begin{center}
	{\includegraphics[scale=.26]{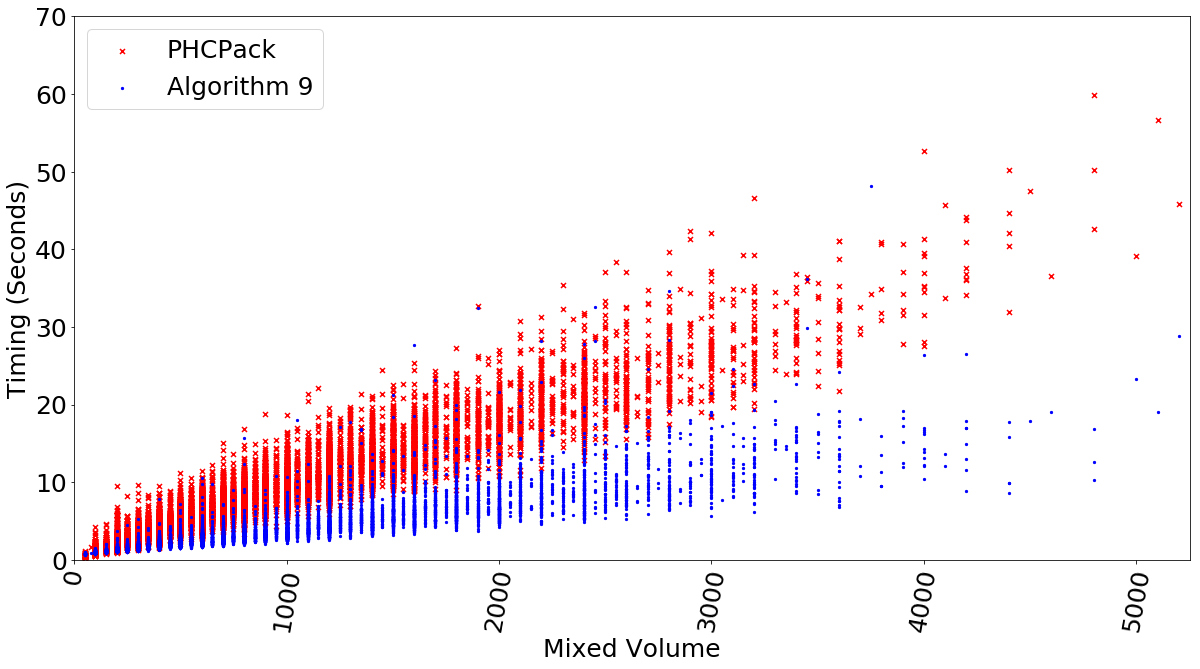}}
   \end{center}
\caption{Scatter plot of timings}
\label{fig:scatter}
\end{figure}
\begin{figure}[htb!]
	\begin{center}
	{\includegraphics[scale=.26]{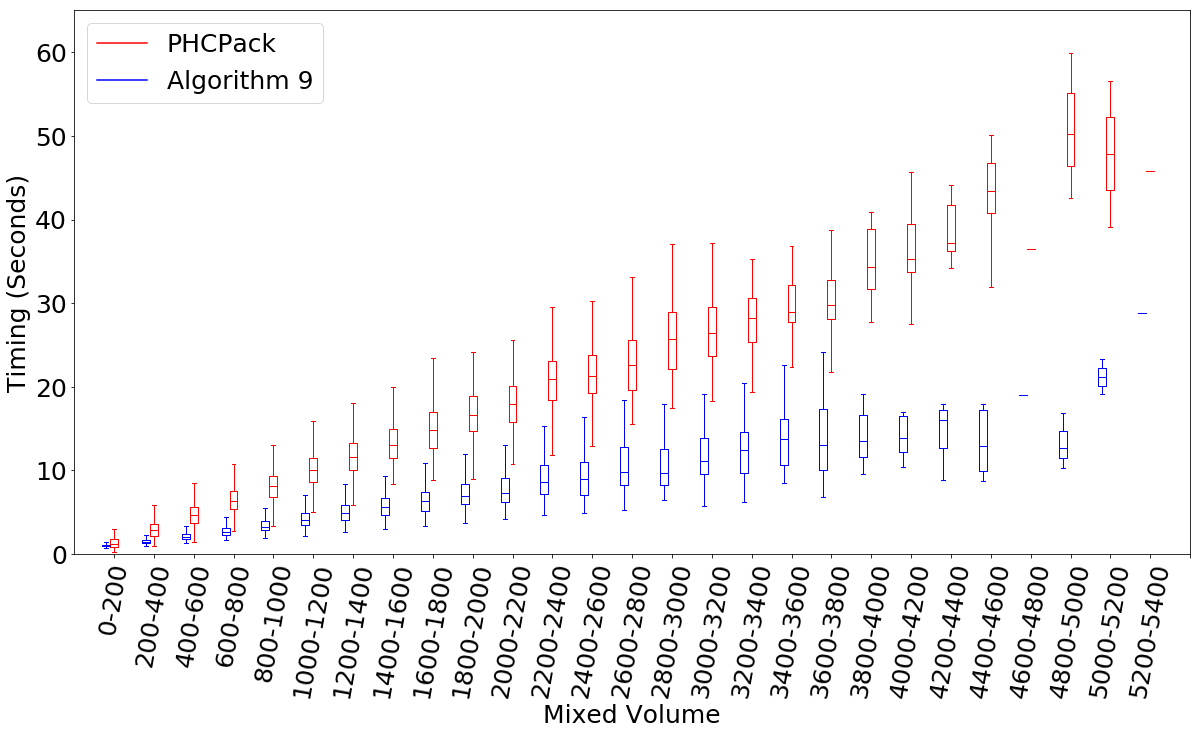}}
   \end{center}
\caption{Box plot of timings}
\label{fig:boxplots}
\end{figure}

A more detailed account of these computations, along with our implementation in Macaulay2, may be found
at the website for this paper~\cite{ThomasWebsite}. 
%
\def\cprime{$'$}
\providecommand{\bysame}{\leavevmode\hbox to3em{\hrulefill}\thinspace}
\providecommand{\MR}{\relax\ifhmode\unskip\space\fi MR }
\providecommand{\MRhref}[2]{%
  \href{http://www.ams.org/mathscinet-getitem?mr=#1}{#2}
}
\providecommand{\href}[2]{#2}

\end{document}